\documentclass{llncs}
\usepackage{llncsdoc}
\usepackage{amsmath,fourier,latexsym,amsfonts,enumerate,listings,dsfont,braket}
\usepackage{graphicx,soul,color}
\newtheorem{algorithm}{Algorithm}
\begin{document}
\pagestyle{empty}

\title{Reparameterization invariant metric\\on the space of curves}

\author{Alice Le Brigant\inst{1,}\inst{2}, Marc Arnaudon\inst{1} \and Fr\'ed\'eric Barbaresco\inst{2}}

\institute{Institut Math\'ematique de Bordeaux, UMR 5251 \\
Universit\'e de Bordeaux and CNRS, France
\and
Thales Air System, Surface Radar Domain, Technical Directorate \\
Voie Pierre-Gilles de Gennes, 91470 Limours, France}

\maketitle
%
%
\begin{abstract}
%
This paper focuses on the study of open curves in a manifold $M$, and proposes a reparameterization invariant metric on the space of such paths. We use the square root velocity function (SRVF) introduced by Srivastava et al. in \cite{sri} to define a reparameterization invariant metric on the space of immersions $\mathcal{M}=\text{Imm}([0,1],M)$ by pullback of a metric on the tangent bundle $\text{T}\mathcal{M}$ derived from the Sasaki metric. We observe that such a natural choice of Riemannian metric on $\text{T}\mathcal{M}$ induces a first-order Sobolev metric on $\mathcal{M}$ with an extra term involving the origins, and leads to a distance which takes into account the distance between the origins and the distance between the SRV representations of the curves. The geodesic equations for this metric are given, as well as an idea of how to compute the exponential map for observed trajectories in applications. This provides a generalized theoretical SRV framework for curves lying in a general manifold $M$.
\end{abstract}
%
%
\section{Introduction}
%
Computing distances between shapes of open or closed curves is of interest in many fields that require shape analysis, from medical imaging to video surveillance, to radar detection. While the shape of an organ or a human contour can be modeled by a closed plane curve, some applications require the manipulation of curves lying in a non flat manifold, such as $S^2$-valued curves representing trajectories on the earth or curves in the space of hermitian positive definite matrices, where the values represent covariance matrices of Gaussian processes. The shape space of planar curves has been widely studied (\cite{mich2},\cite{mich3},\cite{younes},\cite{bauer1}), and the more general setting of shapes lying in any manifold $M$ has recently met great interest (\cite{bauer3},\cite{su},\cite{lb},\cite{zhang}). Here we consider open oriented curves in a Riemannian manifold $M$, more precisely the space of immersions $c : [0,1] \rightarrow M$,
\begin{equation*}
\mathcal{M}=\text{Imm}([0,1],M).
\end{equation*}
Reparameterizations will be represented by increasing diffeomorphisms $\phi : [0,1] \rightarrow [0,1]$ (so that they preserve the end points of the curves), and their set is denoted by $\text{Diff}^+([0,1])$. Then, one way to describe a shape is as the equivalence class of all the curves that are identical modulo reparameterization, and the shape space as the associated quotient space,
\begin{equation*}
\mathcal{S}=\text{Imm}([0,1],M)/\text{Diff}^+([0,1]).
\end{equation*}
The formal principal bundle structure $\pi : \mathcal{M} \rightarrow \mathcal{S}$ induces a decomposition of the tangent bundle $T\mathcal{M}=V\mathcal{M} \oplus H\mathcal{M}$ into a vertical subspace $V\mathcal{M}=\ker(T\pi)$ consisting of all vectors tangent to the fibers of $\mathcal{M}$ over $\mathcal{S}$, and a horizontal subspace $H\mathcal{M}=\left(V\mathcal{M}\right)^{\perp_G}$ defined as the orthogonal complement of $V\mathcal{M}$ according to the metric $G$ that we put on $\mathcal{M}$. We say formal because the manifold structure of the space $\text{Imm}([0,1],M)$ has not yet been thoroughly studied to our knowledge. We require that $G$ be reparameterization invariant, that is to say that the action of $\text{Diff}^+([0,1])$ be isometric for $G$
\begin{equation}
\label{reparinv}
G_{c\circ \phi}(h\circ \phi, k\circ \phi)=G_c(h,k),
\end{equation}
for any curve $c\in \mathcal{M}$, reparameterization $\phi \in \text{Diff}^+([0,1])$, and infinitesimal deformations $h,k \in T_c\mathcal{M}$ -- $h$ and $k$ can also be seen as vector fields along the curve $c$ in $M$. That way, the induced geodesic distance between two curves $c_0$ and $c_1$ does not change if we reparameterize them the same way, that is
\begin{equation*}
d(c_o\circ\phi,c_1\circ\phi) = d(c_0,c_1),
\end{equation*}
for any $\phi\in\text{Diff}^+([0,1])$. Also, if this property is satisfied, then $G$ induces a Riemannian metric $\hat G$ on the shape space,
\begin{equation*}
\hat G_{\pi( c)}\left( T_c\pi(h),T_c\pi(k)\right)=G_c(h^H,k^H),
\end{equation*}
in the sense that the above expression does not depend on the choice of the representatives $c$, $h$ and $k$. Here $h^H, k^H$ denote the horizontal parts of $h$ and $k$ according to the previously mentioned decomposition, as well as the horizontal lifts of $T_c\pi(h)$ and $T_c\pi(k)$, respectively. The geodesic distances $d$ on $\mathcal M$ and $\hat d$ on $\mathcal S$ are then simply linked by
\begin{equation*}
\hat d\left(\, [c_0] \, , \, [c_1] \,\right) = \inf \left\{\, d\left(c_0, c_1\circ \phi \right) \, | \, \, \phi \in \text{Diff}^+([0,1]) \, \right\},
\end{equation*}
where $[c_0]$ and $[c_1]$ denote the shapes of two given curves $c_0$ and $c_1$, and $\hat d$ verifies the stronger property
\begin{equation*}
\hat d(c_0\circ\phi,c_1\circ\psi)=\hat d(c_0,c_1), 
\end{equation*}
for any reparameterizations $\phi,\psi\in\text{Diff}^+([0,1])$. The most natural candidate for a reparameterization invariant metric $G$ on $\mathcal{M}$ is the $L^2$-metric with integration over arc length, but Michor and Mumford have shown in \cite{mich1} that the induced metric $\hat G$ on the shape space always vanishes. This has motivated the study of Sobolev metrics (\cite{mich3},\cite{bauer1},\cite{bauer2}), and particularly of a first-order Sobolev metric on the space of plane curves,
\begin{equation}
\label{sobolev}
G_c(h,k)=\int \langle D_{\ell}h^\perp,D_\ell k^\perp \rangle + \frac{1}{4} \langle D_\ell h^\parallelslant, D_\ell k^\parallelslant \rangle \, \mathrm d\ell,
\end{equation}
where we integrate according to arc length $\mathrm d\ell = \left\|c'(t)\right\| \mathrm dt \,$ and $\langle \cdot,\cdot \rangle$ denotes the euclidean metric on $\bbbr^2$, $\, D_\ell h=\frac{1}{\left\| c' \right\|}h' \,$ is the derivation of $h$ according to arc length, $D_sh^\parallelslant=\langle D_sh,v\rangle v \,$ is the projection of $D_sh$ on the unit length tangent vector field $\, v=\frac{1}{\left\| c' \right\|}c' \,$ along $c$, and $\, D_sh^\perp=\langle D_sh, n\rangle n \,$ is the projection of $D_sh$ on the unit length normal vector field $n$ along $c$. This particular first-order Sobolev metric is of interest because it can be studied via the square root velocity (SRV) framework, introduced by Srivastava et al. in \cite{sri} and used in several applications (\cite{laga},\cite{su}). This framework can be extended to curves in a general manifold by using parallel transport, in a way which allows us to move the computations to the tangent plane to the origin of one of the two curves under comparison, see \cite{lb} and \cite{zhang}. In \cite{lb} the transformation used is a generalization of the SRV function introduced by Bauer et al. in \cite{bauer1} as a tool to study a more general form of the Sobolev metric \eqref{sobolev}. In \cite{zhang} a Riemannian framework is given, including the associated Riemannian metric and the geodesic equations. While our approach in this paper is similar, we feel that the distance we introduce here will be more directly dependent on the "relief" of the manifold, since it is computed in the manifold itself rather than in one tangent plane as in \cite{lb} and \cite{zhang}. This enables us to take into account a greater amount of information on the space separating two curves.
%
%
\section{New metric on the space of parameterized curves}
%
We consider the square root velocity function (SRVF) introduced in \cite{sri} on the space of curves in $M$,
\begin{equation*}
R : \mathcal{M} \rightarrow T\mathcal{M}, \quad c \mapsto \frac{c'}{\sqrt{\left\|c'\right\|}},
\end{equation*}
where $\left\| \cdot \right\|$ is the norm associated to the Riemannian metric on $M$. This function will allow us to define a metric $G$ on $\mathcal{M}$ by pullback of a metric $\tilde G$ on $T\mathcal{M}$. First, we define the following projections from $TTM$ to $T M$. Let $\xi \in T_{(p,u)}T M$ and $(x,U)$ be a curve in $T M$ that passes through $(p,u)$ at time $0$ at speed $\xi$. Then we define the vertical and horizontal projections
\begin{eqnarray*}
\text{vp}_{(p,u)} &:& T_{(p,u)}T M \rightarrow T_pM, \quad \xi \mapsto \xi_V := \nabla_{x'(0)}U, \\
\text{hp}_{(p,u)} &:& T_{(p,u)}T M \rightarrow T_pM, \quad \xi \mapsto \xi_H := x'(0).
\end{eqnarray*}
The horizontal and vertical projections live in the tangent bundle $TM$ and are not to be confused with the horizontal and vertical parts which live in the double tangent bundle $TTM$ and will be denoted by $\xi^H$, $\xi^V$. Furthermore, let us point out that the horizontal projection is simply the differential of the natural projection $T M \rightarrow M$, and that according to these definitions, the Sasaki metric (\cite{sas1}, \cite{sas2}) can be written
\begin{equation*}
g^{S}_{(p,u)}(\xi,\eta) = \Braket{\, \xi_H\, ,\, \eta_H \,} + \Braket{\, \xi_V\, ,\, \eta_V \,},
\end{equation*}
where $\langle \cdot,\cdot \rangle$ denotes the Riemannian metric on $M$. Now we can define the metric that we put on $T\mathcal{M}$. Let us consider $\, h \in T\mathcal{M}\,$ and $\, \xi, \eta \in T_hT\mathcal{M \,}$. We define
\begin{equation}
\tilde G_h\left(\xi,\eta\right) \,= \, \Braket{\, \xi(0)_H\, ,\, \eta(0)_H } \,+\, \int_0^1 \Braket{ \, \xi(t)_V \, , \, \eta(t)_V } \, \mathrm dt,
\end{equation}
where $\, \xi(t)_H=\text{hp}(\xi(t)) \,$ and $\, \xi(t)_V=\text{vp}(\xi(t)) \,$ are the horizontal and vertical projections of  $\, \xi(t) \in TTM \,$ for all $t$. Then we have the following result.
\begin{proposition}
The pullback of the metric $\tilde G$ by the square root velocity function $R$ is given by
\begin{equation}
G_c(h,k) = \Braket{ \, h(0) , k(0) \, } + \int \Braket {\nabla_\ell h^\perp,\nabla_\ell k^\perp }  + \frac{1}{4} \Braket{ \nabla_\ell h^\parallelslant, \nabla_\ell k^\parallelslant } \, \mathrm d\ell,
\end{equation}
for any curve $c\in \mathcal{M}$ and vectors $h,k\in T_c\mathcal{M}$, where we integrate according to arc length, $\nabla_\ell h=\frac{1}{\left\| c' \right\|}\nabla_{c'}h$ is the covariant derivative of $h$ according to arc length, and $\nabla_\ell h^\parallelslant=\langle D_\ell h,v\rangle v \,$ and $\, \nabla_\ell h^\perp=\nabla_\ell h-\nabla_\ell h^\parallelslant$ are its tangential and normal components respectively. 
\end{proposition}
\begin{remark}
In the case of curves in a flat space, $G$ is the first-order Sobolev metric \eqref{sobolev}, studied in \cite{sri}, with an added term involving the origins. This extra term guaranties that the induced distance is always greater than the distance between the starting points of the curves in $M$.
\end{remark}
\begin{proof}
For any $c\in \mathcal{M}$, and $h,k \in T_c\mathcal{M}$, the metric $G$ is defined by
\begin{equation*}
G_c(h,k) = \tilde G_{R( c)}\left( T_cR(h) , T_cR(k) \right).
\end{equation*}
For any $t\in[0,1]$, we have $T_cR(h)(t)_H=h(t)$ and $T_cR(h)_V=\nabla_hR(c )(t)$. To prove this proposition, we just need to compute the latter. Let $\, s \mapsto c(s,\cdot) \,$ be a curve in $\mathcal{M}$ such that $\, c(0,\cdot)=c \,$ and $\, c_s(0,\cdot)=h \,$. Here and in all the paper we use the notations $\, c_s=\partial c/\partial s \,$ and $\, c_t=\partial c/\partial t \,$. Then
\begin{eqnarray*}
\nabla_hR(c )(t)&=& \frac{1}{\left\|c'\right\|^{1/2}} \nabla_h c' + h\left(\left\|c' \right\|^{-1/2}\right)c' \\
&=& \frac{1}{\left\| c_t\right\|^{1/2}} \nabla_{s} c_t+ \partial_s \Braket{\, c_t \, , \, c_t \,}^{-1/4} c_t \\
&=& \frac{1}{\left\| c_t\right\|^{1/2}} \nabla_t c_s - \frac{1}{2} \Braket{ \, c_t\, ,\,c_t\,}^{-5/4} \Braket{ \, \nabla_s c_t \, , \,c_t \,} \, c_t \\
&=& \left\|c'\right\|^{1/2} \left( \left( \nabla_\ell h\right)^\perp + \frac{1}{2} \langle \nabla_\ell h\, ,\, \frac{c'}{\|c'\|}\rangle \frac{c'}{\|c'\|} \right),
\end{eqnarray*}
where in the last step we use again the inversion $\nabla_s c_t=\nabla_tc_s$.
\end{proof}
%
%
\section{Fiber bundle structures}
%
\subsubsection*{Principal bundle over the shape space} We already know that we have a formal principal bundle structure over the shape space
\begin{equation*}
\pi : \mathcal{M}=\text{Imm}([0,1],M) \rightarrow \mathcal{S}=\mathcal{M}/ \text{Diff}^+([0,1]).
\end{equation*}
which induces a decomposition $\, T\mathcal{M}=V\mathcal{M} \overset{\perp}{\oplus} H\mathcal{M} \,$. Just as in the planar case, the fact that the square root velocity function $R$ verifies the equivariance property 
\begin{equation*}
R(c\circ \phi) = \sqrt{\phi'} \left(R(c ) \circ \phi \right) 
\end{equation*}
for all $c \in \mathcal{M}$, $h,k \in T_c\mathcal{M}$ and $\phi \in \text{Diff}^+([0,1])$, guaranties that the integral part of $G$ is reparameterization invariant. Remembering that the reparameterizations $\phi\in \text{Diff}^+([0,1])$ preserve the origins of the curves, we notice that $G$ is constant along the fibers, as expressed in equation \eqref{reparinv}, and so there exists a Riemannian metric $\hat G$ on the shape space $\mathcal{S}$ such that $\pi$ is (formally) a Riemannian submersion from $(\mathcal{M},G)$ to $(\mathcal{S},\hat G)$
\begin{equation*}
G_c(h^H,k^H) = \hat G_{\pi( c)}\left( T_c\pi(h), T_c\pi(k) \right),
\end{equation*}
where $h^H$ and $k^H$ are the horizontal parts of $h$ and $k$ respectively. 
\subsubsection*{Fiber bundle over the starting points}
The special role that plays the starting point in the metric $G$ induces another formal fiber bundle structure, where the base space is the manifold $M$, seen as the set of starting points of the curves, and the fibers are the set of curves with the same origin. The projection is then
\begin{equation*}
\pi^{(*)} : \mathcal{M} \rightarrow M, \quad c \mapsto c(0).
\end{equation*}
It induces another decomposition of the tangent bundle in vertical and horizontal bundles
\begin{eqnarray*}
V^{(*)}_c\mathcal{M} &=& \ker T\pi^{(*)} = \left\{ \, h \in T_c\mathcal{M} \, | \, h(0)=0 \, \right\}, \\
H^{(*)}_c\mathcal{M} &=& \left( V^{(*)}_c\mathcal{M}\right)^{\perp_G}.
\end{eqnarray*}
\begin{proposition}
We have the usual decomposition $T\mathcal{M}=V^{(*)}\mathcal{M} \,\, \overset{\perp}{\oplus} \,\, H^{(*)}\mathcal{M}$, the horizontal bundle $H^{(*)}_c\mathcal{M}$ consists of parallel vector fields along $c$, and $\pi^{(*)}$ is (formally) a Riemannian submersion for $(\mathcal{M},G)$ and $(M, \langle \cdot, \cdot \rangle)$.  
\end{proposition}
\begin{proof}
Let $h$ be a tangent vector. Consider $h_0$ the parallel vector field along $c$ with initial value $h_0(0)=h(0)$. It is a horizontal vector, since its vanishing covariant derivative along $c$ assures that for any vertical vector $l$ we have $G_c(h_0,l)=0$. The difference $\tilde h=h-h_0$ between those two horizontal vectors has initial value $0$ and so it is a vertical vector, which gives a decomposition of $h$ into a horizontal vector and a vertical vector. The definition of $H^{(*)}\mathcal M$ as the orthogonal complement of $V^{(*)}\mathcal M$ guaranties that their sum is direct. Now if $k$ is another tangent vector, then the scalar product between their horizontal parts is
\begin{equation*}
G_c(h^H,k^H)=\Braket{ \, h^H(0)\, ,\,k^H(0) \, }_{c(0)} =\Braket{ \, h(0)\, ,\,k(0) \, }_{c(0)} = \Braket{\, T_c\pi^{(*)}(h^H) \, ,\, T_c\pi^{(*)}(k^H) \,}_{\pi^{(*)}},
\end{equation*}
and this completes the proof.
\end{proof}
%
%
%
\section{Induced distance on the space of curves}
%
Here we will give an expression for the geodesic distance induced by the metric $G$. Let us consider two curves $c_0, c_1 \in \mathcal{M}$, and a path of curves $s \mapsto c(s,\cdot)$ linking them in $\mathcal{M}$ \begin{equation*}
c(0,t)=c_0(t), \quad c(1,t)=c_1(t),
\end{equation*} 
for all $t \in [0,1]$. We denote by $q(s,\cdot)=R\left(c(s,\cdot)\right)$ the image of this path of curves by the SRVF $R$. Note that $q$ is a vector field along the surface $c$ in $M$. Let now $\tilde q$ be the raising of $q$ in the tangent plane $T_{c(0,0)}M$ in the following way
\begin{equation*}
\tilde q(s,t) = P_{c(\cdot,0)}^{s,0}\circ P_{c(s,\cdot)}^{t,0} \left( q(s,t) \right),
\end{equation*}
where we denote by $P_\gamma^{t_1,t_2}:T_{\gamma(t_1)}M \rightarrow T_{\gamma(t_2)}M$ the parallel transport along a curve $\gamma$ from $\gamma(t_1)$ to $\gamma(t_2)$. Notice that $\tilde q$ is a surface in a vector space, as illustrated in Figure \ref{fig}. Lastly, we introduce a vector field $(a,\tau)\mapsto \omega^{s,t}(a,\tau)$ in $M$, which parallel translates $q(s,t)$ along $c(s,\cdot)$ to its origin, then along $c(\cdot,0)$ and back down again, as shown in Figure \ref{fig}. More precisely
\begin{equation*}
\omega^{s,t}(a,\tau)=P_{c(a,\cdot)}^{0,\tau} \circ P_{c(\cdot,0)}^{s,a} \circ P_{c(s,\cdot)}^{t,0} \left( q(s,t) \right)
\end{equation*}
for all $b,s$. That way the quantity $\nabla_s\omega^{s,t}$ measures the holonomy along the rectangle of infinitesimal width shown in Figure \ref{fig}. 
We can now formulate our result.
%
\begin{proposition}
\label{distance1}
With the above notations, the geodesic distance induced by the Riemannian metric $G$ between two curves $c_0$ and $c_1$ on the space $\mathcal{M}=\text{Imm}([0,1],M)$ of parameterized curves is given by
\begin{equation}
\label{eqdist1a}
\text{dist}(c_0,c_1)= \inf_{c(0,\cdot)=c_0, c(1,\cdot)=c_1} \,\,\, \int_0^1 \sqrt{\left\| c_s(s,0) \right\|^2 + \int_0^1 \left\| \nabla_sq(s,t) \right\|^2 \, \mathrm dt} \,\,\, \mathrm ds,
\end{equation}
where $q=R(c)$ is the Square Root Velocity representation of the curve $c$ and the norm is the one associated to the Riemannian metric on $M$. It can also be written
\begin{equation}
\label{eqdist1b}
\text{dist}(c_0,c_1)= \inf_{c(0,\cdot)=c_0, c(1,\cdot)=c_1} \,\,\, \int_0^1 \sqrt{\left\| c_s(s,0) \right\|^2 + \int_0^1 \left\| \, \tilde q_s(s,t) +  \Omega(s,t) \right\|^2 \, \mathrm dt} \,\,\, \mathrm ds,
\end{equation}
where $\tilde q$ is the raising of $q$ in the tangent plane $T_{c(0,0)}M$ and the curvature term $\Omega$ is given by
\begin{eqnarray*}
\Omega(s,t)&=&P_{c(\cdot,0)}^{s,0} \circ P_{c(s, \cdot)}^{t,0} \left(\nabla_s\omega^{s,t}(s,t) \right)\\
&=& P_{c(\cdot,0)}^{s,0} \circ P_{c(s, \cdot)}^{t,0} \left( \int_0^t P_{c(s,\cdot)}^{\tau,t}\left( \mathcal R( c_\tau, c_s) P_{c(s,\cdot)}^{t,\tau}q(s,t) \right) \, \mathrm d\tau\right),
\end{eqnarray*}
if $\mathcal R$ denotes the curvature tensor of the manifold $M$.
\end{proposition}
%
\begin{remark}
Our original motivation for this work was to find a geodesic distance (that is, a distance induced by a Riemannian metric) that resembled the product distance introduced in \cite{lb}. In the first term under the square root of expression \eqref{eqdist1b} we can see the velocity vector of the curve $c(\cdot,0)$ linking the two origins, and in the second the velocity vector of the curve $\tilde q$ linking the TSRVF-images of the curves -- Transported Square Root Velocity Function, as introduced by Su et al. in \cite{su}.  However there is also a curvature term $\Omega$ which, as previously mentionned, measures the holonomy along the rectangle of infinitesimal width shown in Figure \ref{fig}. If instead we equip the tangent bundle $\text{T}\mathcal{M}$ with the metric
\begin{equation*}
\tilde G_h(\xi,\xi) = \left\| \xi_h(0) \right\|^2 + \int_0^1 \left\| \, \xi_v(t) - \int_0^t P_c^{\tau,t}\left( \mathcal R(c', \xi_h) P_c^{t,\tau}h(t) \right)\, \mathrm d\tau \, \right\|^2 \mathrm dt,
\end{equation*}
for $h \in T\mathcal{M}$ and $\, \xi, \eta \in T_hT\mathcal{M}$, then the curvature term $\Omega$ vanishes and the geodesic distance on $\mathcal{M}$ becomes
\begin{equation}
\label{eqdist2}
\text{dist}(c_0,c_1)=\inf_{c(0,\cdot)=c_0,c(1,\cdot)=c_1} \,\,\,\, \int_0^1 \sqrt{\left\| c_s(s,0) \right\|^2 + \left\| \tilde q_s(s,\cdot) \right\|^2 } \,\,\, \mathrm ds,
\end{equation}
where the norm of the second term under the square root is the $L^2$-norm, and which corresponds exactly to the geodesic distance associated to the metric on the space $\mathbb{C}=\cup_{p\in M} L^2([0,1],T_pM)$ introduced by Zhang et al. in \cite{zhang}. The difference between the two distances \eqref{eqdist1a} and \eqref{eqdist2} resides in the curvature term $\Omega$, which translates the fact that in the first one, we compute the distance in the manifold, whereas in the second, it is computed in the tangent space to one of the origins of the curves. Therefore, the first one takes more directly into account the "relief" of the manifold between the two curves under comparison. For example, if there is a "bump" between two curves in an otherwise relatively flat space, the second distance \eqref{eqdist2} might not see it, whereas the first one \eqref{eqdist1a} will thanks to the curvature term.
\end{remark}
%
\begin{remark}
Let us briefly consider the flat case : if the manifold $M$ is flat, the two distances \eqref{eqdist1a} and \eqref{eqdist2} coincide. 
If two curves $c_0$ and $c_1$ in a flat space have the same starting point $p$, the first summand under the square root vanishes and the distance becomes the $L^2$-distance between the two SRV representations $q_0=R(c_0)$ and $q_1=R(c_1)$. If two curves in a flat space differ only by a translation, then the distance is simply the distance between their origins.
\end{remark}
\begin{figure}[h]
\centering
\includegraphics[width=8.2cm]{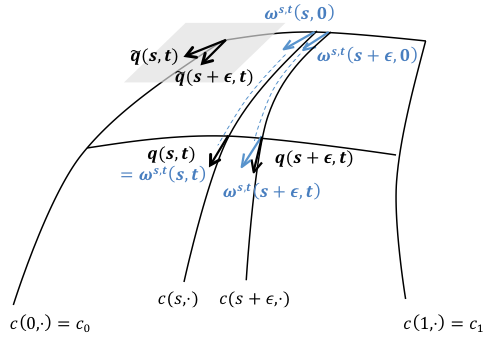}
\caption{Illustration of the distance between two curves $c_0$ and $c_1$ in the space of curves $\mathcal{M}$}
\label{fig}
\end{figure} 
\begin{proof}
Since $G$ is defined by pullback of $\tilde G$ by the SRVF $R$, we know that the lengths of $c$ in $\mathcal{M}$ and of $q=R( c)$ in $\text{T}\mathcal{M}$ are equal and so that
\begin{equation*}
\text{dist}(c_0,c_1)=\inf_{c(0,\cdot)=c_0,c(1,\cdot)=c_1} \,\,\, \int_0^1 \sqrt{\tilde G \left( q_s(s,\cdot),q_s(s,\cdot)\right)} \,\,\, \mathrm ds,
\end{equation*}
with 
\begin{equation*}
\tilde G \left( q_s(s,\cdot),q_s(s,\cdot)\right) = \left\| c_s(s,0) \right\|^2 + \int_0^1 \left\| \nabla_sq(s,t) \right\|^2 \, \mathrm dt.
\end{equation*}
Now let us fix $t \in [0,1]$. Then $s \mapsto P_{c(s,\cdot)}^{t,0}\left(q(s,t)\right)$ is a vector field along $c(\cdot, 0)$, and so
\begin{equation*}
\nabla_s\left( P_{c(s,\cdot)}^{t,0}q(s,t) \right)=P_{c(\cdot, 0)}^{0,s} \left( \frac{\partial}{\partial s} P_{c(\cdot, 0)}^{s,0} \circ P_{c(s,\cdot)}^{t,0} \left(q(s,t)\right) \right) = P_{c(\cdot, 0)}^{0,s} \left( \tilde q_s (s,t) \right).
\end{equation*}
We consider the vector field $\nu$ along the surface $(s,\tau)\mapsto c(s,\tau)$ that is parallel along all curves $c(s,\cdot)$ and takes value $\nu(s,t)=q(s,t) \,$ in $\tau=t \,$ for any $s\in[0,1]$, that is
\begin{equation*}
\nu(s,\tau)=P_{c(s,\cdot)}^{t,\tau}\left(q(s,t)\right),
\end{equation*}
for all $s\in [0,1]$ and $\tau\in [0,1]$. That way we know that 
\begin{eqnarray*}
\nabla_s\nu(s,t)&=&\nabla_sq(s,t),  \\
\nabla_s\nu(s,0)&=& P_{c(\cdot, 0)}^{0,s} \left( \tilde q_s(s,t) \right), \\ 
\nabla_\tau\nu(s,\tau)&=&0,  
\end{eqnarray*}
for all $s,\tau \in[0,1]$. Then we can express its covariant derivative in the following way
\begin{eqnarray}
\label{nablav}
\nabla_s\nu(s,t)&=&P_{c(s,\cdot)}^{0,t} \left( \nabla_s\nu(s,0) \right) + \int_0^t P_{c(s,\cdot)}^{\tau,t}\left( \nabla_\tau\nabla_s \nu(s,\tau) \right) d\tau \notag \\
&=& P_{c(s,\cdot)}^{0,t} \circ P_{c(\cdot, 0)}^{0,s} \left( \tilde q_s(s,t) \right) + \int_0^t P_{c(s,\cdot)}^{\tau,t}\left( \mathcal R(c_\tau,c_s) P_{c(s,\cdot)}^{t,\tau}q(s,t) \right) \, \mathrm d\tau.
\end{eqnarray}
Now let us fix $s\in [0,1]$ as well. Notice that the vector field $\omega^{s,t}$ defined above verifies
\begin{eqnarray*}
\omega^{s,t}(s,t)&=&q(s,t), \\ 
\nabla_\tau\omega^{s,t}(a,\tau)&=& 0, \\ 
\nabla_a\omega^{s,t}(a,0)&=& 0,
\end{eqnarray*}
for all $a,\tau \in [0,1]$. Note that unlike $\nu$, we do \textit{not} have $\nabla_s\omega^{s,t}(s,t)=\nabla_sq(s,t)$ because $\omega^{s,t}(a,t)=q(a,t)$ is only true for $a=s$. It is easy to verify that the last term of equation \eqref{nablav} is precisely the covariant derivative of the vector field $\omega^{s,t}$
\begin{equation*}
\nabla_s\omega^{s,t}(s,t) = \int_0^t P_{c(s,\cdot)}^{\tau,t}\left( \mathcal R(c_\tau, c_s) P_{c(s,\cdot)}^{t,\tau}q(s,t) \right) \, \mathrm d\tau,
\end{equation*}
since for any $\tau\in[0,1]$, $\omega^{s,t}(s,\tau)=P_{c(s,\cdot)}^{t,\tau}q(s,t)$, and finally by composing by $P_{c(\cdot,0)}^{s,0} \circ P_{c(s, \cdot)}^{t,0}$, we obtain the second expression \eqref{eqdist1b}, which completes the proof.
\end{proof}
%
%
%
\section{Geodesic equation on $\mathcal M$}
%
In order to be able to compute the distance given by \eqref{eqdist1a} between two curves, we first need to compute the optimal deformation $s\mapsto c(s,\cdot)$ from one to the other. In other words, we need to characterize the geodesics of $\mathcal M$ for our metric. In order to do so, we use the variational principle. The beginning of the calculations are very similar to those in \cite{zhang}. Let us consider two curves $c_0,c_1\in\mathcal M$ and a path $[0,1]\ni s \mapsto c(s,\cdot)\in\mathcal M$ going from one to the other. This path $c$ is a geodesic if it minimizes the energy functional $E : \mathcal M \rightarrow \mathbb R_+$
\begin{equation*}
E(c) = \int_0^1 G\left(\frac{\partial c}{\partial s},\frac{\partial c}{\partial s}\right) \mathrm ds.
\end{equation*}
Let $a\mapsto c(a,\cdot,\cdot)$, $a\in(-\epsilon, \epsilon)$, be a proper variation of the path $s\mapsto c(s,\cdot)$, meaning that it coincides with $c$ in $a=0$, and it preserves its end points
\begin{eqnarray*}
c(0,\cdot,\cdot)&=&c, \\
c_a(a,0,\cdot)&=&0 \quad \forall a\in(-\epsilon,\epsilon),\\
c_a(a,1,\cdot)&=&0 \quad \forall a\in(-\epsilon,\epsilon).
\end{eqnarray*}
Then $c$ is a geodesic of $\mathcal M$ if $\left. \frac{d}{da}\right|_{a=0}E(c(a,\cdot,\cdot))=0$ for any proper variation $a\mapsto c(a,\cdot,\cdot)$. If we denote by $E(a) = E(c(a,\cdot,\cdot))$, for $a\in(-\epsilon,\epsilon)$, the energy of a proper variation of $c$, then we have
\begin{equation*}
E(a)=\int \Braket{\,   c_s(a,s,0),   c_s(a,s,0) \,} \mathrm ds \, + \, \int\int \Braket{\, \nabla_s  q(s,t),\nabla_s   q(s,t) \,} \mathrm dt \, \mathrm ds,
\end{equation*}
where $q=c_t/\sqrt{\|c_t\|}$ is the SRV representation of $c$. Its derivative is given by
\begin{equation*}
\frac{1}{2} E'(a) = \int \, \Braket{ \, \nabla_a  c_s(a,s,0),   c_s(a,s,0) \,} \, \mathrm ds 
+ \int\int \, \Braket{\, \nabla_a \nabla_s   q(a,s,t),\nabla_s  q(a,s,t) \,} \, \mathrm dt \, \mathrm ds. 
\end{equation*}
Considering that the variation preserves the end points, integration by parts gives
\begin{eqnarray*}
\int \Braket{\, \nabla_a  c_s,  c_s \,} \mathrm ds &=& - \int \Braket{\nabla_s   c_s,   c_a} \mathrm ds \\
\int \Braket{\, \nabla_s\nabla_a  q, \nabla_s q \,} \mathrm ds &=& - \int \Braket{\nabla_s \nabla_s q, \nabla_a q} \mathrm ds,
\end{eqnarray*}
and so we obtain
\begin{eqnarray*}
\dfrac{1}{2}E'(a)&=& - \, \int \left. \Braket{ \, \nabla_s   c_s,   c_a \,} \right|_{t=0} \mathrm ds
       \, + \, \int \int \Braket{ \, \mathcal R (  c_a,   c_s)   q \, +\, \nabla_s\nabla_a  q , \nabla_s  q\,} \mathrm dt \, \mathrm ds \\
       &=& - \, \int \left. \Braket{ \, \nabla_s   c_s ,   c_a \,} \right|_{t=0} \mathrm ds
       \,-\, \int \int \Braket{\, \mathcal R(   q, \nabla_s   q)  c_s ,   c_a \,} \, + \, \Braket{\, \nabla_s \nabla_s q , \nabla_a q\,} \mathrm dt \, \mathrm ds.
\end{eqnarray*}
This quantity has to vanish in $a=0$ for all proper variations 
\begin{equation*}
\int \Braket{ \, \left. \nabla_s c_s \right|_{t=0}, \left.   c_a \right|_{a=0, t=0} \, } \mathrm ds \nonumber 
+ \int \int \Braket{\, \mathcal R(v,\nabla_sq)c_s,\left.   c_a \right|_{a=0} \,} \, + \, \Braket{\, \nabla_s\nabla_s q, \left. \nabla_a  q \right|_{a=0} \,} \, \mathrm dt \, \mathrm ds \, = \, 0. \nonumber 
\end{equation*} 
Unfortunately, we cannot yield any conclusions at this point, because $c_a(0,s,t)$ and $\nabla_aq(0,s,t)$ cannot be chosen independently, since $q$ is not any vector field along $c$ but its image via the Square Root Velocity Function. Using the notations
\begin{eqnarray*}
B(s,t) &=& \mathcal R(q,\nabla_s q)c_s (s,t), \\
D(s,t) &=& \frac{1}{\sqrt{|c_t|}} \nabla_s\nabla_s q(s,t) \, - \, \frac{1}{2} \frac{\Braket{\, \nabla_s\nabla_s q \, , \, c_t \,}}{|c_t|^{5/2}}c_t(s,t),
\end{eqnarray*}
and the lighter notation $u(t_1)^{t_1,t_2}=P_c^{t_1,t_2}(u(t_1))$ to denote the parallel transport along a curve $c$ of a vector field $u$, the following simple manipulations provide us with a solution
\begin{eqnarray*}
&& \int_0^1 \Braket{ \nabla_s c_s(s,0), c_a(0,s,0) } \mathrm ds + \int_0^1 \int_0^1 \Braket{ B(s,t), c_a(0,s,t)} + \Braket{D(s,t),\nabla_t  c_a(0,s,t)}\mathrm dt \, \mathrm ds \\
&=& \int_0^1 \Braket{ \nabla_s c_s(s,0), c_a(0,s,0) } \mathrm ds + \int_0^1 \int_0^1 \Braket{ B(s,t), c_a(0,s,0)^{0,t}+\int_0^t\nabla_tc_a(0,s,\tau)^{\tau,t}\mathrm d\tau} \mathrm dt \, \mathrm ds \\
&+& \int_0^1\int_0^1 \Braket{D(s,t),\nabla_t c_a(0,s,t)} \, \mathrm dt \, \mathrm ds \\
&=&  \int_0^1 \Braket{ \nabla_s c_s(s,0) + \int_0^1 B(s,t)^{t,0}\mathrm dt, c_a(0,s,0) } \mathrm ds + \int_0^1\int_0^1\int_0^t \Braket{B(s,t)^{t,\tau},\nabla_tc_a(0,s,\tau)} \mathrm d\tau \, \mathrm dt \, \mathrm ds \\
&+& \int_0^1\int_0^1 \Braket{D(s,t),\nabla_t c_a(0,s,t)} \, \mathrm dt \, \mathrm ds \\
&=& \int_0^1 \Braket{ \nabla_s c_s(s,0) + \int_0^1 B(s,t)^{t,0}\mathrm dt \, , \, c_a(0,s,0) } \mathrm ds \\
&+& \int_0^1\int_0^1 \Braket{D(s,t) + \int_t^1 B(s,\tau)^{\tau,t} \mathrm d\tau \, , \, \nabla_t c_a(0,s,t)} \, \mathrm dt \, \mathrm ds.
\end{eqnarray*}
Since the variations $c_a(0,s,0)$ and $\nabla_tc_a(0,s,\cdot)$ can be chosen independently for all $s$, we obtain the following characterization of the geodesics $s\mapsto c(s,\cdot)$ of $\mathcal M$
\begin{eqnarray*}
\nabla_sc_s(s,0) + \int_0^1 \mathcal R(q,\nabla_sq)c_s(s,t)^{t,0} \mathrm dt &=& 0, \quad \forall s\\
D(s,t) + \int_t^1 \mathcal R(q,\nabla_sq)c_s(s,\tau)^{\tau,t} \mathrm d\tau &=& 0, \quad \forall t,s.
\end{eqnarray*}
Finally, using the definition of $D(s,t)$, the geodesic equations are
\begin{subequations}
\begin{align}
\nabla_sc_s(s,0) \,\,\, =& \,\,\, r(s,0) , \quad \forall s \label{geodeq1} \\
\nabla_s\nabla_s q(s,t) \,\,\, =& \,\,\, \left\|q(s,t) \right\| \left( r(s,t) + r(s,t)^\parallelslant \right), \quad \forall t,s \label{geodeq2}
\end{align}
\end{subequations}
where $q$ is the SRV representation of $c$, the vector field $r$ is given by
\begin{eqnarray*}
r(s,t) &=& - \int_t^1 \mathcal R(q,\nabla_sq)c_s(s,\tau)^{\tau,t} \mathrm d\tau,
\end{eqnarray*} 
and $r^{\parallelslant} = \Braket{r,v} v \,\,\,$ with $v=\frac{1}{\|c_t\|}c_t \,\,\,$, is the tangential component of $r$.
%
%
%
\section{Exponential map}
%
In this section, we describe an algorithm which allows us to compute the geodesic $s\mapsto c(s,\cdot)$ starting from a point $c\in \mathcal M$ at speed $u\in T_c\mathcal M$. This amounts to finding the optimal deformation of the curve $c$ in the direction of the vector field $u$ according to our metric. We initialize this path $s\mapsto c(s,\cdot)$ by setting $c(0,\cdot)=c$ and $c_s(s,0)=u$, and we propagate it using iterations of fixed step $\epsilon >0$. The aim is, given $c(s,\cdot)$ and $c_s(s,\cdot)$, to deduce $c(s+\epsilon,\cdot)$ and $c_s(s+\epsilon,\cdot)$. The first is obtained by following the exponential map on the manifold $M$ 
\begin{equation*}
c(s+\epsilon,\cdot)=\exp_{c(s,\cdot)} \epsilon c_s(s,\cdot),
\end{equation*}
and the second requires the computation of the variation $\nabla_s c_s(s,\cdot)$
\begin{equation*}
c_s(s+\epsilon,\cdot)=\left[ c_s(s,\cdot) + \epsilon \nabla_s c_s(s,\cdot) \right]^{s,s+\epsilon},
\end{equation*}
where once again, we use the notation $w(s)^{s,s+\epsilon}=P_c^{s,s+\epsilon} \left(w(s)\right)$ for the parallel transport of a vector field $s\mapsto w(s)$ along a curve $s\mapsto c(s)$ in $M$. If we suppose that at time $s$ we have $c(s,\cdot)$ and $c_s(s,\cdot)$, then we also know $c_t(s,\cdot)$ and $q(s,\cdot)=\frac{1}{\sqrt{\|c_t\|}}c_t(s,\cdot)$, as well as $\nabla_tc_s(s,\cdot)$ and 
\begin{equation}
\label{nablasq}
\nabla_sq(s,\cdot)=\frac{\nabla_s c_t}{\sqrt{|c_t|}}(s,\cdot) - \frac{1}{2} \frac{\Braket{\, \nabla_s c_t, c_t \,}}{|c_t|^{5/2}}c_t(s,\cdot),
\end{equation}
using the fact that $\nabla_s c_t=\nabla_tc_s$. The variation $\nabla_sc_s(s,\cdot)$ can then be computed in the following way
\begin{equation*}
\nabla_sc_s(s,t)=\nabla_sc_s(s,0)^{0,t} + \int_0^t \left[ \nabla_s\nabla_sc_t(s,\tau) - \mathcal R(c_t,c_s)c_s(s,\tau) \right]^{\tau,t} \mathrm d\tau
\end{equation*}
for all $t\in[0,1]$, where $\nabla_sc_s(s,0)$ is given by equation \eqref{geodeq1}, the second order variation $\nabla_s\nabla_sc_t(s,\cdot)$ is given by
\begin{eqnarray}
\label{nablasnablasct}
\nabla_s \nabla_s c_t &=& |c_t|^{1/2} \nabla_s\nabla_sq \,+\, \frac{ \Braket{\nabla_tc_s,c_t}}{|c_t|^2} \nabla_tc_s \nonumber \\
&+& \left( \frac{\Braket{\nabla_s\nabla_sq,c_t}}{|c_t|^{3/2}} \,-\, \frac{3}{2} \frac{\Braket{\nabla_tc_s,c_t}^2}{|c_t|^4} \,+\, \frac{|\nabla_tc_s|^2}{|c_t|^2} \right) c_t,
\end{eqnarray}
and $\nabla_s\nabla_sq$ can be computed via equation \eqref{geodeq2}.

Now, if we no longer have a continuous curve but a series of discrete observations $p_0, p_1, \hdots, p_n \in M$ made at discrete times $0=t_0 < t_1< \hdots < t_n=1$, then the optimal deformation of these points in the direction of a series of tangent vectors $u_0, u_1, \hdots, u_n$, where $u_k \in T_{p_k}M\,\,$ for $\,\,k=0,\hdots,n$, can be computed by the following steps.
%
\begin{algorithm}[Discrete Exponential Map] 
\leavevmode\par \noindent
Initialization : Set $c(0,t_k)=p_k$ and $c_s(0,t_k)=u_k$ for $k=0,\hdots,n$. \\
Heredity : If $c(s,t_k)$ and $c_s(s,t_k)$ are known for all $k$, then
\begin{enumerate} 
\item for all $\,\,0\leq k\leq n-1$, set
\begin{eqnarray*}
c_t(s,t_k) &=& \frac{1}{t_{k+1}-t_k} \log _{c(s,t_k)} c(s,t_{k+1}), \\
\nabla_tc_s(s,t_k) &=& \frac{1}{t_{k+1}-t_k} \left( c_s(s,t_{k+1})^{t_{k+1},t_k} - c_s(s,t_k) \right),
\end{eqnarray*}
where $\,\, \log \,\,$ denotes the inverse of the exponential map on $M$, and compute $q(s,t_k)=\frac{1}{\sqrt{\|c_t\|}}c_t(s,t_k)$ and $\nabla_sq(s,t_k)$ using equation \eqref{nablasq}.
\vspace{0.5em}
\item Compute $r(s,t_k)=\sum_{\ell=k}^n \mathcal R(q,\nabla_sq)c_s(s,t_\ell)^{t_\ell,t_k}$ for all $\,\,0\leq k\leq n$ and set
\begin{eqnarray*}
\nabla_sc_s(s,t_0) &=& r(s,t_0), \\
\nabla_s\nabla_s q(s,t_k) &=& \|q(s,t_k)\| \left( r(s,t_k) + r(s,t_k)^\parallel \right).
\end{eqnarray*}
\item For $k=0,\hdots,n-1$,
\begin{enumerate}
\item compute $\nabla_s\nabla_sc_t(s,t_k)$ using equation \eqref{nablasnablasct},
\item compute $\nabla_s c_s(s,t_{k+1})$ using 
\begin{equation*}
\nabla_s c_s(s,t_{k+1}) = \left[\,\, \nabla_s c_s(s, t_k) \, + \, (t_{k+1}-t_k) \left(\, \nabla_s\nabla_s c_t(s,t_k)\,-\, \mathcal R(c_s,c_t)c_s(s,t_k)\, \right) \,\, \right]^{t_k,t_{k+1}}.
\end{equation*}
\end{enumerate}
\item Finally, for all $0\leq k \leq n$, set
\begin{eqnarray*}
c(s+\epsilon,t_k) &=& \exp_{c(s,t_k)} \left( \epsilon c_s(s,t_k) \right) \\
c_s(s+\epsilon,t_k) &=& \left[ \,\,c_s(s,t_k) + \epsilon \nabla_sc_s(s,t_k) \,\,\right]^{t_k,t_{k+1}}.
\end{eqnarray*}
\end{enumerate}
\end{algorithm}
%
%
%
\section{Conclusion}
%
In the same way that the first-order Sobolev metric \eqref{sobolev} on the space of plane curves can be obtained as the pullback of the $L^2$-metric by the square root velocity function (\cite{sri}), our metric $G$ can be obtained as the pullback of a natural metric $\tilde G$ on the tangent bundle $\text{T}\mathcal{M}$ by the same SRVF. As such it is reparameterization invariant, and induces a Riemannian metric $\hat G$ on the shape space $\mathcal{S}$ for which the fiber bundle projection is formally a Riemannian submersion. On the other hand, the special role that $G$ gives to the starting points of the curves induces another formal fiber bundle structure, this time over the manifold $M$ seen as the set of starting points of the curves, for which the projection is formally also a Riemannian submersion. The geodesic distance induced by $G$ takes into account the distance between the origins of the curve and the $L^2$-distance between the $SRV$ representations, but instead of transporting the computations in a unique tangent plane as in \cite{lb} and \cite{zhang}, we stay in the manifold. This should allow us to take into account a greater amount of information on its geometry. Finally, explicit equations can be obtained for the geodesics on the space $\mathcal M$ of parameterized curves for our metric $G$, and the exponential map can be iteratively computed. Future work will include some numerical simulations to illustrate the work of this paper. 
%
%
%
\section*{Acknowledgments}
%
 This research was supported by Thales Air Systems and the french MoD DGA.
%
%
%


\begin{thebibliography}{99}
%
\bibitem{bauer1}
M. Bauer, M. Bruveris, S. Marsland, and P. W. Michor. 
Constructing reparameterization invariant metrics on spaces of plane curves. 
Differential Geom. Appl., 34:139-165, 2014.
%
\bibitem{bauer2}
M. Bauer, M. Bruveris, P. Michor. 
Why use Sobolev metrics on the space of curves, 
To appear in Riemannian Computing in Computer Vision (Springer). arXiv:1502.03229.
%
\bibitem{bauer3}
M. Bauer, P. Harms, P. Michor.
Sobolev metrics on shape space of surfaces.
Journal of Geometric Mechanics 3, 4, 389-438, 2011.
%
\bibitem{laga}
H. Laga, S. Kurtek, A. Srivastava, S. J. Miklavcic. 
Landmark-free statistical analysis of the shape of plant leaves. 
J. Theoret. Biol., 363:41-52, 2014.
%
\bibitem{lb}
A. Le Brigant, M. Arnaudon, F. Barbaresco.
Reparameterization invariant distance on the space of curves in the hyperbolic plane.
AIP Conf. Proc. 1641, 504, 2015.
%
\bibitem {mich1}
P. Michor, D. Mumford. 
Vanishing geodesic distance on spaces of submanifolds and diffeomorphisms. 
Documenta Math. 10, 217-245, 2005.
%
\bibitem{mich2}
P. Michor, D. Mumford. 
Riemannian geometries on spaces of plane curves. 
J. Eur. Math. Soc. (JEMS) 8, 1-48, 2006.
%
\bibitem{mich3}
P. Michor, D. Mumford. 
An overview of the Riemannian metrics on spaces of curves using the Hamiltonian approach. 
Applied and Computational Harmonic Analysis 23, 74-113, 2007.
%
\bibitem{sas1}
S. Sasaki. 
On the differential geometry of tangent bundles of Riemannian manifolds.
Tohoku Math. J. 10, 338-354, 1958.
%
\bibitem{sas2}
S. Sasaki.
On the differential geometry of tangent bundles of Riemannian manifolds
II, Tohoku Math. J. 14, 146-155, 1962.
%
\bibitem{sri}
A. Srivastava, E. Klassen, S. H. Joshi, and I. H. Jermyn. Shape analysis of elastic curves in Euclidean spaces. IEEE T. Pattern Anal., 33(7):1415-1428, 2011.
%
\bibitem{su}
J. Su, S. Kurtek, E. Klassen, A. Srivastava. 
Statistical analysis of trajectories on Riemannian manifolds: bird migration, hurricane tracking and video surveillance. 
Ann. Appl. Stat., 8(1):530-552, 2014.
%
\bibitem{younes}
L. Younes, P. Michor, J. Shah, D. Mumford. 
A Metric on Shape Space with Explicit Geodesics. 
Rend. Lincei Mat. Appl. 9, 25-57, 2008.
%
\bibitem{zhang}
Z. Zhang, J. Su, E. Klassen, H. Le, A. Srivastava.
Video-based action recognition using rate-invariant analysis of covariance trajectories.
arXiv:1503.06699 [cs.CV], 2015.
%
\end{thebibliography}
\end{document}